\documentclass[12pt]{amsart}
\usepackage{amssymb,amsmath}
\usepackage{geometry}    
\usepackage{mathrsfs}

\usepackage{vmargin}
\setmargrb{1in}{1in}{1in}{1in}

\newtheorem{theorem}{Theorem}[section]
\newtheorem{lemma}[theorem]{Lemma}

\newtheorem{corollary}[theorem]{Corollary}
\theoremstyle{definition}

\newtheorem{remark}[theorem]{Remark}

\newtheorem{problem}[theorem]{Problem}

\begin{document}

\title{On uniform approximation to successive powers of a real number}

\author{Johannes Schleischitz} 

\address{Institute of Mathematics, Boku Vienna, Austria  \\ 
johannes.schleischitz@boku.ac.at}

\begin{abstract}
We establish new inequalities involving classical exponents of Diophantine approximation.
This allows for improving on the work of H. Davenport, W.M. Schmidt and M. Laurent concerning the maximum value
of the exponent $\widehat{\lambda}_{n}(\zeta)$ among all real transcendental $\zeta$.
In particular we refine the estimation $\widehat{\lambda}_{n}(\zeta)\leq \lceil n/2\rceil^{-1}$
due to M. Laurent by
$\widehat{\lambda}_{n}(\zeta)\leq \widehat{w}_{\lceil n/2\rceil}(\zeta)^{-1}$ for all $n\geq 1$,
and for even $n$ we replace the bound $2/n$ for $\widehat{\lambda}_{n}(\zeta)$ first found by Davenport
and Schmidt by roughly $\frac{2}{n}-\frac{4}{n^{3}}$, which provides the currently best known bound when $n\geq 6$.
\end{abstract}

\maketitle

{\footnotesize{Supported by the Austrian Science Fund FWF grant P24828.} \\

{\em Keywords}: simultaneous approximation, geometry of numbers, successive minima \\
Math Subject Classification 2010: 11H06, 11J13, 11J25}

\section{Introduction}

\subsection{Exponents of approximation} \label{sek1}

We start with the definition of classical exponents of Diophantine approximation. 
Let $\zeta$ be a real transcendental number and $n\geq 1$ be an integer. For $1\leq j\leq n+1$
we define the exponents of simultaneous approximation $\lambda_{n,j}(\zeta)$ as
the supremum of $\eta\in{\mathbb{R}}$ such that the system 
\begin{equation}  \label{eq:lambda}
\vert x\vert \leq X, \qquad \max_{1\leq i\leq n} \vert \zeta^{i}x-y_{i}\vert \leq X^{-\eta},  
\end{equation}
has (at least) $j$ linearly independent solutions $(x,y_{1},y_{2},\ldots, y_{n})\in{\mathbb{Z}^{n+1}}$ 
for arbitrarily large values of $X$. Moreover, let $\widehat{\lambda}_{n,j}(\zeta)$
be the supremum of $\eta$
such that \eqref{eq:lambda} has (at least) $j$ linearly independent solutions for all $X\geq X_{0}$. 
In case of $j=1$ we also only write $\lambda_{n}(\zeta)$ and $\widehat{\lambda}_{n}(\zeta)$ respectively,
which coincide with the classical exponents of approximation defined by Bugeaud and Laurent~\cite{buglau}.
By Dirichlet's Theorem, for all $n\geq 1$ and transcendental real $\zeta$, these exponents are
bounded below by
\begin{equation} \label{eq:ldiri}
\lambda_{n}(\zeta)\geq \widehat{\lambda}_{n}(\zeta)\geq \frac{1}{n}.
\end{equation}
Moreover it is clear from the definition that
\begin{equation}  \label{eq:ldirich}
\lambda_{1}(\zeta)\geq \lambda_{2}(\zeta)\geq \cdots, 
\qquad \widehat{\lambda}_{1}(\zeta)\geq \widehat{\lambda}_{2}(\zeta)\geq \cdots.
\end{equation}

Similarly,  for $1\leq j\leq n+1$ let $w_{n,j}(\zeta)$ and $\widehat{w}_{n,j}(\zeta)$ be 
the supremum of $\eta\in{\mathbb{R}}$ such that the system 
\begin{equation}  \label{eq:w}
H(P) \leq X, \qquad  0<\vert P(\zeta)\vert \leq X^{-\eta},  
\end{equation}
has (at least) $j$ linearly independent polynomial solutions 
$P(T)=a_{n}T^{n}+a_{n-1}T^{n-1}+\cdots+a_{0}$ of degree at most $n$
with integers $a_{j}$ for arbitrarily large $X$ and all large $X$, respectively. Here 
$H(P)=\max_{0\leq j\leq n} \vert a_{j}\vert$ is the height of $P$ as usual. 
Again for $j=1$ we also write $w_{n}(\zeta)$ and $\widehat{w}_{n}(\zeta)$,
which coincide with classical exponents. Dirichlet's Theorem implies the estimates
\begin{equation} \label{eq:wmono}
w_{n}(\zeta)\geq \widehat{w}_{n}(\zeta)\geq n.
\end{equation}
Moreover it is obvious that the exponents satisfy the relations
\begin{equation} \label{eq:wmonos}
w_{1}(\zeta)\leq w_{2}(\zeta)\leq \cdots, 
\qquad \widehat{w}_{1}(\zeta)\leq \widehat{w}_{2}(\zeta)\leq \cdots. 
\end{equation}
We point out that identities due to Mahler can be stated in the form
\begin{align}
\lambda_{n,j}(\zeta)= \frac{1}{\widehat{w}_{n,n+2-j}(\zeta)}, \qquad\qquad 1\leq j\leq n+1,  \label{eq:aligneins} \\
w_{n,j}(\zeta)= \frac{1}{\widehat{\lambda}_{n,n+2-j}(\zeta)}, \qquad\qquad 1\leq j\leq n+1,  \label{eq:alignzwei}
\end{align}
as carried out in~\cite{j2}. In particular \eqref{eq:alignzwei} will be of major importance
for us, since it is essential in the proof of our main result Theorem~\ref{neu}. 
The remaining results we now state below
again concern the best approximation constants, i.e. $j=1$ in the above definitions.
For the remainder of this section all results are implicitly understood to be
valid for all $n\geq 1$ and all real transcendental $\zeta$, unless stated otherwise. 
We remark that some of the results can be extended to the case of vectors 
$\underline{\zeta}\in\mathbb{R}^{n}$ which are $\mathbb{Q}$-linearly independent together with $\{1\}$.
Khintchine's transference inequalities~\cite{khintchine} connects the exponents 
of best approximation with the polynomial approximation constants via 
\begin{equation} \label{eq:khintchine}
\frac{w_{n}(\zeta)}{(n-1)w_{n}(\zeta)+n}\leq \lambda_{n}(\zeta)\leq \frac{w_{n}(\zeta)-n+1}{n}.
\end{equation}
German~\cite{german} established inequalities involving the uniform constants given by
\begin{equation} \label{eq:ogerman}
\frac{\widehat{w}_{n}(\zeta)-1}{(n-1)\widehat{w}_{n}(\zeta)}\leq
 \widehat{\lambda}_{n}(\zeta)\leq \frac{\widehat{w}_{n}(\zeta)-n+1}{\widehat{w}_{n}(\zeta)}. 
\end{equation}
For $n=2$ this reproved Jarn\'iks identity
\begin{equation} \label{eq:jahrnik}
\widehat{\lambda}_{2}(\zeta)=1-\frac{1}{\widehat{w}_{2}(\zeta)}.
\end{equation}
Concerning $n=2$, Davenport and Schmidt~\cite{davsh} further showed 
\begin{equation} \label{eq:najo}
\widehat{w}_{2}(\zeta)\leq \frac{3+\sqrt{5}}{2}=:\gamma\approx 2.6180.
\end{equation}
Equality can be obtained and Roy called such numbers extremal numbers, see~\cite{royyy}. 
For an exponent of approximation we call its spectrum the set of values attained by inserting
the set of real transcendental numbers.
Among other things it is known that the spectrum
of $\widehat{w}_{2}$ in the interval $[2,\gamma]$ is dense, see~\cite{royexp}, and uncountable~\cite{buglau}.
For $n\geq 3$ the best known upper bound for $\widehat{w}_{n}(\zeta)$ is given by
\begin{equation} \label{eq:bugschl}
\widehat{w}_{n}(\zeta)\leq n-\frac{1}{2}+\sqrt{n^{2}-2n+\frac{5}{4}},
\end{equation}
which is the main result of~\cite{buschl}. The right hand side is of order $2n-3/2+o(1)$ for 
large $n$ and always smaller than the bound $2n-1$ known before by Davenport and Schmidt~\cite{davsh}.
For $n=3$ the stronger bound
\begin{equation} \label{eq:stronger}
\widehat{w}_{3}(\zeta)\leq 3+\sqrt{2}\approx 4.4142
\end{equation}
was also established in~\cite[Theorem~2.1]{buschl}. 
We do not even know whether $\widehat{w}_{n}(\zeta)>n$ can occur for $n\geq 3$ and real transcendental $\zeta$.

Concerning the simultaneous approximation problem, Davenport and Schmidt~\cite{davsh} were the first 
to establish upper bounds for $\widehat{\lambda}_{n}(\zeta)$.
Laurent~\cite{laurent} refined their result for odd $n\geq 5$. The best upper bounds 
evolving from the mentioned papers can be stated in the form
\begin{equation} \label{eq:laurent}
\widehat{\lambda}_{n}(\zeta)\leq \frac{1}{\left\lceil \frac{n}{2}\right\rceil}, \qquad n\geq 1.
\end{equation}
For even $n$, the resulting bound $2/n$ actually due to Davenport and Schmidt,
will be refined in Theorem~\ref{lauverb}.
For $n=2$, from \eqref{eq:jahrnik} and \eqref{eq:najo} follows the stronger estimate
\begin{equation} \label{eq:zwei}
\widehat{\lambda}_{2}(\zeta)\leq \frac{\sqrt{5}-1}{2}:=\sigma\approx 0.6180.
\end{equation}
Jarn\'ik's identity \eqref{eq:jahrnik} moreover implies that the bound is sharp as well
and that the spectrum of $\widehat{\lambda}_{2}$ is dense in $[1/2,\sigma]$ as well.
For $n=3$, Roy~\cite{damroy} improved the bound $1/2$ resulting from \eqref{eq:laurent} to
\begin{equation} \label{eq:roybound}
\widehat{\lambda}_{3}(\zeta)\leq \frac{2+\sqrt{5}-\sqrt{7+2\sqrt{5}}}{2}=: \rho\approx 0.4245.
\end{equation}
We will refer to the values $\gamma,\sigma, \rho$ above in the following sections.
It is well-known and follows from \eqref{eq:ogerman} that $\widehat{w}_{n}(\zeta)>n$ is
equivalent to $\widehat{\lambda}_{n}(\zeta)>1/n$, and again we do not know whether
$\widehat{\lambda}_{n}(\zeta)>1/n$ can be obtained for any real transcendental $\zeta$ when $n\geq 3$.
For the proof of Theorem~\ref{lauverb} we will need the relation
\begin{equation} \label{eq:umform}
\lambda_{n}(\zeta)\geq 
\frac{\widehat{\lambda}_{n}(\zeta)^{2}+(n-2)\widehat{\lambda}_{n}(\zeta)}{(n-1)(1-\widehat{\lambda}_{n}(\zeta))},
\end{equation}
valid for all real transcendental $\zeta$ and $n\geq 2$, due to Schmidt and Summerer~\cite[(1.21)]{ssch}. 
See also~\cite{sums} for an improvement when $n=3$, where also a stronger estimate 
in any dimension $n\geq 4$ was conjectured.

Finally we introduce the exponents of approximation by algebraic numbers.
Let $w_{n}^{\ast}(\zeta)$ be the supremum of $\eta\in\mathbb{R}$ such that
\[
0<\vert \zeta-\alpha\vert \leq H(\alpha)^{-\eta}
\]
has a real algebraic solution $\alpha$ of degree at most $n$. Here $H(\alpha)=H(P)$ is the height
of the minimal polynomial $P$ of $\alpha$ over $\mathbb{Z}[T]$ with coprime coefficients. The
uniform exponent $\widehat{w}_{n}^{\ast}(\zeta)$ is defined as the supremum of $\eta$ such that
\[
H(\alpha)\leq X, \qquad 0<\vert \zeta-\alpha\vert \leq H(\alpha)^{-1}X^{-\eta}
\]
has a solution as above for all large $X$. Again the estimates
\[
\cdots \geq w_{2}^{\ast}(\zeta)\geq w_{1}^{\ast}(\zeta), \qquad
\cdots \geq \widehat{w}_{2}^{\ast}(\zeta)\geq \widehat{w}_{1}^{\ast}(\zeta)
\]
and 
\[
\widehat{w}_{n}^{\ast}(\zeta)\leq w_{n}^{\ast}(\zeta), \qquad n\geq 1,
\]
are obvious. Moreover we know that
\begin{equation} \label{eq:sternnicht}
w_{n}(\zeta)-n+1\leq w_{n}^{\ast}(\zeta)\leq w_{n}(\zeta), \qquad 
\widehat{w}_{n}(\zeta)-n+1\leq \widehat{w}_{n}^{\ast}(\zeta)\leq \widehat{w}_{n}(\zeta)
\end{equation}
see~\cite[Lemma~A.8]{bugbuch}. It is a longstanding open conjecture due 
to Wirsing that $w_{n}^{\ast}(\zeta)\geq n$ for all $n\geq 1$ and any transcendental real $\zeta$. We know that
\begin{equation} \label{eq:ber}
w_{n}^{\ast}(\zeta)\geq \frac{n+\sqrt{n^{2}+16n-8}}{4}
\end{equation}
by Bernik and Tishchenko~\cite{bertis}, which is of order $n/2+2+o(1)$ for large $n$. 
This has been further improved by Tishchenko~\cite{tish} to an expression of order
$n/2+3-o(1)$ for large $n$, which is currently the best known  bound.
We will need two estimates involving exponents of simultaneous approximation and algebraic approximation.
For any $n\geq 1$ and transcendental real $\zeta$, Davenport and Schmidt~\cite[Lemma~1]{davsh} showed the estimate
\begin{equation} \label{eq:schmidt}
w_{n}^{\ast}(\zeta)\geq \frac{1}{\widehat{\lambda}_{n}(\zeta)}.
\end{equation}
Similarly we have
\begin{equation} \label{eq:ichdeb}
\widehat{w}_{n}^{\ast}(\zeta)\geq \frac{1}{\lambda_{n}(\zeta)},
\end{equation}
which is the main result of~\cite{j2}. Further
we notice that due to a transference argument of Davenport and Schmidt~\cite{davsh},
at least in \eqref{eq:schmidt}
we may choose algebraic integers of degree $n+1$ instead of the algebraic numbers of degree $n$
within the definition of the exponent $w_{n}^{\ast}$. By this we mean that for any $\epsilon>0$ the inequality
\[
\vert \zeta-\alpha\vert \leq H(\alpha)^{-1/\lambda_{n}(\zeta)-1+\epsilon}
\]
has infinitely many algebraic integer solutions $\alpha$ of degree $n+1$. 

\subsection{Mahler's classification} \label{mala}
We end the introduction with Mahler's classification of real transcendental numbers, which will
appear at some places in the paper. For $m\geq 1$ an integer, 
a real transcendental number $\zeta$ is called a $U_{m}$-number 
if $w_{m}(\zeta)=\infty$ and $m$ is the smallest such index. The set of $U$-numbers is the 
union of the sets of $U_{m}$-numbers over $m\geq 1$. The set of $U_{1}$-numbers is also called Liouville numbers.
The existence of $U_{m}$-numbers for arbitrary $m\geq 1$ was established by LeVeque~\cite{leveque}.
A real transcendental number is called a $T$-number if $\limsup_{n\to\infty} w_{n}(\zeta)/n=\infty$ but
$w_{n}(\zeta)<\infty$ for all $n\geq 1$. Schmidt~\cite{tnumbers} showed the existence of $T$-numbers.
Finally a transcendental real number $\zeta$ is called an $S$-number if $\limsup_{n\to\infty} w_{n}(\zeta)/n<\infty$.
It is easy to see that the classes $S,T,U$ form a partition of the set of real transcendental numbers.

\section{New results} \label{newr} 

\subsection{Upper bounds for $\widehat{\lambda}_{n}$} \label{hauptres}

The main and most general result of the paper links certain exponents from Section~\ref{sek1}
in different dimensions.

\begin{theorem} \label{neu}
Let $m,n$ be positive integers and $\zeta$ be a real transcendental number. Then 
\begin{equation}  \label{eq:glm}
\widehat{\lambda}_{m+n-1}(\zeta)\leq \max \left\{ \frac{1}{w_{m}(\zeta)},\frac{1}{\widehat{w}_{n}(\zeta)}\right\}.
\end{equation}
\end{theorem}

In fact every result in the present Section~\ref{hauptres}
will be a rather immediate consequence of Theorem~\ref{neu}, partly in combination 
with some known results from Section~\ref{sek1}.
First observe that the claim of the theorem 
can be reformulated
\[
\min\{ w_{m}(\zeta),\widehat{w}_{n}(\zeta) \} \leq \frac{1}{\widehat{\lambda}_{m+n-1}(\zeta)}.
\]
In view of \eqref{eq:ldiri}, this refines~\cite[Theorem 2.3]{buschl} which asserted 
\begin{equation} \label{eq:unte}
\min\{ w_{m}(\zeta),\widehat{w}_{n}(\zeta) \} \leq m+n-1.
\end{equation}
Notice that if $m=n$, 
in view of the trivial relations \eqref{eq:ldiri} and \eqref{eq:wmono}, we infer the following.

\begin{corollary} \label{neues}
Let $n$ be a positive integer and $\zeta$ be a real transcendental number. Then
\[
\widehat{\lambda}_{n}(\zeta)\leq \frac{1}{\widehat{w}_{\lceil n/2\rceil}(\zeta)}.
\]
\end{corollary}

Again from \eqref{eq:wmono} we see that Corollary~\ref{neues}
confirms Laurent's estimate \eqref{eq:laurent}. In fact our proof of Theorem~\ref{neu}
will be simpler and shorter than the one in~\cite{laurent},
and relies on a duality argument between the simultaneous approximation problem and the
polynomial approximation problem, hidden in \eqref{eq:alignzwei} employed in the proof. 
Although the methods differ essentially, the polynomials
$P_{i}$ defined on page 49 in~\cite{laurent} derived from the best approximation vectors
are connected with the polynomials implied by the definition of $\widehat{w}_{n}(\zeta)$ 
involved in Theorem~\ref{neu}. See also Remark~\ref{schauher}.
On the other hand, in the ''usual'' case $\widehat{w}_{n}(\zeta)=n$, 
Corollary~\ref{neues} does not allow for improving \eqref{eq:laurent} for a given $n$. 
However, in Theorem~\ref{lauverb} below we will find an improvement for even $n$ via combination 
with \eqref{eq:umform}.

\begin{theorem} \label{lauverb}
Let $n$ be a positive integer and define the cubic polynomial
\begin{equation} \label{eq:polynom}
P_{n}(T)=T^{3}+\frac{n-1}{n}T^{2}+(2n-1)T+\frac{-2n+1}{n}.
\end{equation}
There exists a unique real root $\theta_{n}$ of $P_{n}$
in the interval $(0,1/n)$, and for $\zeta$ any real transcendental number we have
\[
\widehat{\lambda}_{2n}(\zeta)\leq \theta_{n}.
\]
\end{theorem} 

It can be shown that for large $n$ the bound $\theta_{n}$ is of order 
\[
\theta_{n}=\frac{1}{n}-\frac{1}{2n^{3}}-\frac{1}{4n^{4}}+O(n^{-5}).
\]
Hence the improvement to \eqref{eq:laurent} is only of order $n^{-3}$. However,
Theorem~\ref{lauverb} yields the first improvement of \eqref{eq:laurent} and the first 
improvement for even $n$ of Davenport and Schmidt's result~\cite{davsh} from 1969.
Cardano's formula yields a closed expression for $\theta_{n}$ in terms of $n$ only, 
however it is too complicated to be useful.
For small $n$ we calculate the first few bounds
\[
\widehat{\lambda}_{2}(\zeta)\leq \theta_{1}\approx 0.6823, \;
\widehat{\lambda}_{4}(\zeta)\leq \theta_{2}\approx 0.4395, \;
\widehat{\lambda}_{6}(\zeta)\leq \theta_{3}\approx 0.3140, \; 
\widehat{\lambda}_{8}(\zeta)\leq \theta_{4}\approx 0.2417.
\]
For $n=1$ the bound is relatively close to the optimal value $\sigma$, see Section~\ref{sek1}.
It should also be pointed out that the bound for $\widehat{\lambda}_{4}(\zeta)$ is significantly
stronger than $1/2$ from \eqref{eq:laurent}, even though
$\widehat{\lambda}_{4}(\zeta)\leq \widehat{\lambda}_{3}(\zeta)\leq \rho$
inferred from \eqref{eq:roybound} is still slightly better. 
Thus Theorem~\ref{lauverb} yields a refinement of the best known upper 
bound for $\widehat{\lambda}_{n}$ for even $n\geq 6$. Finally we remark
that the claim of Theorem~\ref{lauverb} could be slightly improved assuming the conjectured
improved estimates from~\cite{sums} for the discrepancy between $\lambda_{n}(\zeta)$
and $\widehat{\lambda}_{n}(\zeta)$ by Schmidt and Summerer. Assuming the conjecture is true, 
with the formula~\cite[(30)]{ichglasgow} some
computation performed with Mathematica yields approximately the first few bounds
\[
\widehat{\lambda}_{2}(\zeta)\leq 0.6823, \qquad
\widehat{\lambda}_{4}(\zeta)\leq 0.4292, \qquad \widehat{\lambda}_{6}(\zeta)\leq 0.3084,
\qquad \widehat{\lambda}_{8}(\zeta)\leq 0.2387.
\]
For $n=2$ the bounds naturally coincide because in this case \eqref{eq:umform} equals 
the estimate for the regular graph. 
Recall that $n=4$ is the smallest index for which the conjecture has not been verified yet,
and note that the bound for $\widehat{\lambda}_{4}$ is still larger than $\rho$ from \eqref{eq:roybound}.
We expect by no means that the bound in Theorem~\ref{lauverb} is optimal for any $n$.

Theorem~\ref{neu}, together with \eqref{eq:ogerman}, gives some information on 
the combined spectra of $(\widehat{\lambda}_{n})_{n\geq 1}$
and $(\widehat{w}_{n})_{n\geq 1}$. We only state the consequence of the case $m=n$ explicitly.

\begin{corollary} \label{ko}
Let $n\geq 2$ be an integer and $\zeta$ be a real transcendental number. Then
\begin{equation}  \label{eq:puha}
\widehat{\lambda}_{2n-1}(\zeta)\leq \frac{1-\widehat{\lambda}_{n}(\zeta)}{n-1}
\end{equation}
and
\begin{equation} \label{eq:har}
\widehat{w}_{2n-1}(\zeta)\leq \frac{\widehat{w}_{n}(\zeta)}{\widehat{w}_{n}(\zeta)-2n+2}, \qquad\quad
\text{if} \;\; \widehat{w}_{n}(\zeta)>2n-2.
\end{equation}
\end{corollary}

The estimate \eqref{eq:puha} shows that $\widehat{\lambda}_{n}(\zeta)$ and $\widehat{\lambda}_{2n-1}(\zeta)$
cannot both be very close to the respective upper bound in \eqref{eq:laurent}. 
Observe that this claim does not follow from Theorem~\ref{lauverb} when $n$ is odd.
However, the improvement
is at most of order $n^{-2}$. For every $n$, the estimate \eqref{eq:puha} is stronger than the trivial
estimate $\widehat{\lambda}_{2n-1}(\zeta)\leq \widehat{\lambda}_{n}(\zeta)$ from \eqref{eq:ldirich} when
$\widehat{\lambda}_{n}(\zeta)>1/n$ strictly.
We highlight the case $n=2$ of \eqref{eq:puha}, where we obtain
\begin{equation} \label{eq:dreii}
\widehat{\lambda}_{2}(\zeta)+\widehat{\lambda}_{3}(\zeta) \leq 1.
\end{equation}
This estimate slightly improves the trivial bound $\sigma+\rho\approx 1.0425$ from
combination of \eqref{eq:zwei} and \eqref{eq:roybound}.
Recall that the spectrum of $\widehat{\lambda}_{2}$ is dense 
in $[1/2,\sigma]$ as mentioned in Section~\ref{sek1}, so there exist many numbers for which \eqref{eq:dreii}
is relevant. From the existence of extremal numbers,
for which $\widehat{\lambda}_{2}(\zeta)=\sigma$ holds, and \eqref{eq:ldiri},
it is also clear that $\sigma+1/3\approx 0.9513$ is the optimal bound one can find in \eqref{eq:dreii}
uniformly for all $\zeta$.
On the other hand, for extremal numbers we have $\widehat{\lambda}_{3}(\zeta)=1/3$, 
as was shown in~\cite{ichlondon}. The same conclusion holds more generally for all numbers $\zeta$ given by 
Sturmian continued fractions (which share the property $\widehat{w}_{2}(\zeta)>2$ see~\cite{buglau}), 
see~\cite{ichstw}. However, there exist examples of numbers that
satisfy $\widehat{\lambda}_{2}(\zeta)>1/2$, but whose continued fraction expansion is
not given by a Sturmian word. In fact any $\zeta$ which is no extremal number
and satisfies $\widehat{\lambda}_{2}(\zeta)>2-\sqrt{2}$ cannot
be given by a Sturmian continued fraction, see~\cite{cassaigne} and also~\cite{fisch1},\cite{fisch2},~\cite{buglau}. 
Since $1-(2-\sqrt{2})=\sqrt{2}-1\approx 0.4142<\rho$, the bound for
$\widehat{\lambda}_{3}(\zeta)$ inferred from \eqref{eq:dreii} improves \eqref{eq:roybound} in particular for any 
such number. 

Now we turn to \eqref{eq:har}.
The claim is most interesting for $n=2$. The condition in \eqref{eq:har} is then just $\widehat{w}_{2}(\zeta)>2$.
We again obtain some improvement for the bound for $\widehat{w}_{3}(\zeta)$ in \eqref{eq:stronger}
if $\widehat{w}_{2}(\zeta)$ is sufficiently close to $\gamma$, more
precisely when $\widehat{w}_{2}(\zeta)>(6+2\sqrt{2})/(2+\sqrt{2})\approx 2.5858$,
which is slightly smaller than $\gamma$. The best upper bound
for $\widehat{w}_{3}$ we can get this way is $2+\sqrt{5}\approx 4.2361$ (for extremal numbers), 
which happens to coincide with the value $w_{2}(\zeta)$ for extremal numbers $\zeta$.
However, for extremal numbers $\zeta$ again we know $\widehat{w}_{3}(\zeta)=3$ by~\cite{ichlondon}.
We further remark that $w_{2}(\zeta)>4$ implies the estimate
$\widehat{w}_{3}(\zeta)\leq 4$ inferred from~\cite[Theorem~2.3]{buschl}, or Theorem~\ref{neu}.
It is not known whether $\widehat{w}_{2}(\zeta)>2$ implies $w_{2}(\zeta)>4$, however this is true
for Sturmian continued fractions where in fact $w_{2}(\zeta)\geq 2+\sqrt{5}$ holds, with equality 
(only) for extremal numbers.
For $n\geq 3$, only provided that $\widehat{w}_{n}$ is very close to the bound in \eqref{eq:bugschl},
the claim \eqref{eq:har} yields an improvement to the bounds arising 
from \eqref{eq:bugschl} in dimension $2n-1$. 
However, we do not believe that the bounds in \eqref{eq:bugschl} are accurate
and in fact there is much doubt whether the condition
in \eqref{eq:har} can be satisfied at all for any $n\geq 3$.
Assuming a conjecture of Schmidt and Summerer~\cite{sums}, the condition always fails when $n\geq 10$,
see~\cite[Theorem~3.3]{ichglasgow}. 

Next we consider the case where $\zeta$ is a $U$-number, see Section~\ref{mala}.

\begin{corollary}
Let $m\geq 1$ be an integer and $\zeta$ be a $U_{m}$-number. Then  we have
\begin{align}
\widehat{\lambda}_{n}(\zeta)&\leq \frac{1}{m}, \qquad\qquad \quad\; m\leq n\leq 2m-1. \label{eq:a1} \\
\widehat{\lambda}_{n}(\zeta)&\leq \frac{1}{n-m+1}, \qquad n\geq 2m-1.  \label{eq:a2}
\end{align}
In particular we have $\widehat{\lambda}_{n}(\zeta)=1/n+O(n^{-2})$.
\end{corollary}

\begin{proof}
It was proved in~\cite[Corollary~2.5]{buschl} that a $U_{m}$-number satisfies $\widehat{w}_{m}(\zeta)=m$, and
thus $\widehat{\lambda}_{m}(\zeta)=1/m$. By the monotonicity \eqref{eq:ldirich} we infer \eqref{eq:a1}.
The stronger bounds \eqref{eq:a2} for $n\geq 2m-1$ follow from Theorem~\ref{neu} and \eqref{eq:wmono}
if we identify $m+n-1$ with the current $n$, such that $n$ in Theorem~\ref{neu}
corresponds to $n-m+1$ in the present notation. 
\end{proof}

\begin{remark} \label{germanchen}
German's transference inequality \eqref{eq:ogerman} is too weak to infer upper bounds for $\widehat{w}_{n}$
from \eqref{eq:a2}. For any $U_{m}$ number $\zeta$ we know the upper bounds
\begin{align} 
\widehat{w}_{n}(\zeta)&\leq m, \qquad\qquad\qquad 1\leq n\leq m,    \label{eq:huchen}     \\
\widehat{w}_{n}(\zeta)&\leq m+n-1, \qquad\; n\geq 1,           \label{eq:kuchen}
\end{align}
which is stronger than \eqref{eq:bugschl} in particular for $n\geq m+1$. 
Indeed \eqref{eq:huchen} follows from $\widehat{w}_{m}(\zeta)=m$ observed 
in the proof above and \eqref{eq:wmonos}, whereas \eqref{eq:kuchen} follows from
combination of $\widehat{w}_{n}^{\ast}(\zeta)\leq m$ for $U_{m}$ numbers, which is~\cite[Corollary~2.5]{buschl}, 
and \eqref{eq:sternnicht}. Conversely, the bound \eqref{eq:kuchen} in combination
with \eqref{eq:ogerman} only yields $\widehat{\lambda}_{n}(\zeta)\leq m/(n+m-1)$, 
which can be verified to be at most as strong as \eqref{eq:a1}, \eqref{eq:a2} for any choice $n\geq m\geq 1$,
and in fact even weaker than Laurent's uniform bound \eqref{eq:laurent} when $n\geq 5$ and $m\geq 3$.
\end{remark}

Observe that $m=1$ reproves the result already noticed in~\cite[Corollary~5.2]{schlei}
that Liouville numbers satisfy $\widehat{\lambda}_{n}(\zeta)=1/n$ 
for all $n\geq 1$.
We can readily combine Theorem~\ref{neu} with \eqref{eq:schmidt} to obtain the following
result, which we want to state as a theorem.  

\begin{theorem} \label{sternchen}
Let $m,n$ be positive integers and $\zeta$ be a real transcendental number. Then 
\[
w_{m+n-1}^{\ast}(\zeta)\geq \min\{ w_{m}(\zeta), \widehat{w}_{n}(\zeta)\}.
\]
In particular
\begin{equation}  \label{eq:glmned}
w_{n}^{\ast}(\zeta)\geq \widehat{w}_{\lceil n/2\rceil}(\zeta).
\end{equation}
The involved $\alpha$ from the definition of $w_{.}^{\ast}$ can be replaced by algebraic integers 
of degree increased by one in the sense of Section~\ref{sek1}. 
\end{theorem}

Again \eqref{eq:glmned} does not yield any unconditioned improvement for the constant $w_{n}^{\ast}$
and is only of interest for rather large values of $\widehat{w}_{n}(\zeta)$.
Since $\widehat{w}_{n}(\zeta)\leq 2n-1$ in fact it does not imply Wirsing's conjecture for any $\zeta$.
It confirms Wirsing's estimate $w_{n}^{\ast}(\zeta)\geq (n+1)/2$ for odd $n$,
which he deduced from the stronger inequality $w_{n}^{\ast}(\zeta)\geq (w_{n}(\zeta)+1)/2$. For even $n$
it yields the slightly weaker bound $w_{n}^{\ast}(\zeta)\geq n/2$. We have discussed
the best currently known lower bounds for $w_{n}^{\ast}$ in Section~\ref{sek1}.  

To finish this section we want to present a condition upon which a slightly 
stronger estimate, involving higher successive minima, than in Theorem~\ref{neu} holds, when $m=n$. 
The proof is very similar to the one of Theorem~\ref{neu} and will be carried out in Section~\ref{proofs}.

\begin{theorem} \label{conditio}
Let $n\geq 1$ be an integer and $\zeta$ be a real transcendental number. 
Assume that
\begin{equation} \label{eq:bedin}
w_{n-1}(\zeta)< w_{n,2}(\zeta). 
\end{equation}
Then 
\begin{equation}  \label{eq:glmneda}
\widehat{\lambda}_{2n-1}(\zeta)\leq \frac{1}{w_{n,2}(\zeta)}.
\end{equation}
\end{theorem} 

Note that for any transcendental real $\zeta$ of
and any positive integer $n$ we have
\begin{equation} \label{eq:fritz}
w_{n,2}(\zeta)\geq \widehat{w}_{n}(\zeta).
\end{equation}
This is a consequence of the fact that two successive best approximations are clearly linearly independent.
More generally $w_{n,j+1}(\zeta)\geq \widehat{w}_{n,j}(\zeta)$ for $1\leq j\leq n$
if $\zeta$ is not algebraic of degree $n$ or less, which follows immediately from~\cite[Theorem~1.1]{ss}. 
Hence \eqref{eq:glmneda} is indeed stronger than \eqref{eq:glm}. Moreover \eqref{eq:fritz} shows that the condition 
\eqref{eq:bedin} is weaker than
\begin{equation} \label{eq:wika}
w_{n-1}(\zeta)< \widehat{w}_{n}(\zeta),
\end{equation}
which is in turn satisfied if $w_{n-1}(\zeta)< n$ in view of \eqref{eq:wmono}. In general we expect strict inequality
in \eqref{eq:fritz}. On the other hand equality can occur, non-trivial
examples are given by extremal numbers $\zeta$ and $n=2$.
Clearly the analogue of Theorem~\ref{sternchen} holds under the condition \eqref{eq:bedin} for the same reason.
Finally we remark that Theorem~\ref{conditio} is in general false without
the condition \eqref{eq:bedin}, since we can even have $w_{n,2}(\zeta)=\infty$, any
Liouville number $\zeta$ has this property, but by \eqref{eq:ldiri} the left hand side in \eqref{eq:glmneda} is at least $(2n-1)^{-1}$.

\subsection{Upper bounds for $\lambda_{n}$}

We want to establish similar bounds for the exponents $\lambda_{n}(\zeta)$. 
The exponents $\lambda_{n}$ may take the value $+\infty$, which is true (precisely) for all
Liouville numbers $\zeta$, see~\cite[Corollary 1]{bug}. Thus we have to make some assumptions. Concretely 
we will need the technical condition \eqref{eq:wika} to guarantee the coprimeness of certain polynomials 
in the proofs. Our most general result is the following.

\begin{theorem} \label{bedthm}
Let $n\geq 2$ be an integer and $\zeta$ be a real transcendental number.
Assume \eqref{eq:wika} holds. Let $m$ be another positive integer. Then
\[
\lambda_{m+n-1}(\zeta)\leq \max\left\{\frac{w_{n}(\zeta)}{\widehat{w}_{m}(\zeta)\widehat{w}_{n}(\zeta)},
\frac{1}{\widehat{w}_{n}(\zeta)}\right\}.
\]
\end{theorem}

The estimate complements the unconditioned bound
\begin{equation} \label{eq:schalecht}
\lambda_{m+n-1}(\zeta)\leq \lambda_{n}(\zeta)\leq \frac{w_{n}(\zeta)-n+1}{n}
\end{equation}
from Khintchine's transference inequalities \eqref{eq:khintchine} and \eqref{eq:ldirich}. Of course it would be 
interesting to know whether the assumption \eqref{eq:wika} can be dropped, or at least weakened. 
We discuss the cases $m\leq n$ and $m>n$ separately. We start with the case $m\leq n$.

\begin{theorem} \label{refthm}
Keep the assumptions and notation of Theorem~\ref{bedthm} and assume $m\leq n$.
Then
\begin{equation} \label{eq:hoemsch}
\lambda_{m+n-1}(\zeta)\leq \frac{w_{n}(\zeta)}{\widehat{w}_{m}(\zeta)\widehat{w}_{n}(\zeta)}.
\end{equation}
In particular
\begin{equation} \label{eq:speziale}
\lambda_{n}(\zeta)\leq \frac{w_{n}(\zeta)}{\widehat{w}_{n}(\zeta)},
\end{equation}
and
\begin{equation} \label{eq:spezial}
\lambda_{2n-1}(\zeta)\leq \frac{w_{n}(\zeta)}{\widehat{w}_{n}(\zeta)^{2}}.
\end{equation}
\end{theorem}

A minor calculation verifies that \eqref{eq:spezial}
provides a proper improvement (upon the assumption \eqref{eq:wika}) of 
\eqref{eq:schalecht} with $m=n$ unless $w_{n}(\zeta)=n$.
It was shown in~\cite[Theorem~2.2]{buschl} that $w_{n-1}(\zeta)<w_{n}(\zeta)$ implies
\begin{equation} \label{eq:aufmucken}
w_{n}(\zeta)\leq (n-1)\frac{\widehat{w}_{n}(\zeta)}{\widehat{w}_{n}(\zeta)-n}.
\end{equation}
Obviously the condition is weaker than \eqref{eq:wika}.
Hence incorporating \eqref{eq:aufmucken} 
in Theorem~\ref{bedthm} we infer bounds which depend solely on uniform constants.

\begin{theorem} \label{juck}
Keep the notation and assumptions of Theorem~\ref{bedthm} and assume $m\leq n$.
Then 
\[
\lambda_{m+n-1}(\zeta)\leq \frac{n-1}{(\widehat{w}_{n}(\zeta)-n)\widehat{w}_{m}(\zeta)}.
\]
In particular
\begin{equation} \label{eq:jucken}
\lambda_{2n-1}(\zeta)\leq \frac{n-1}{(\widehat{w}_{n}(\zeta)-n)\widehat{w}_{n}(\zeta)}.
\end{equation}
\end{theorem}

\begin{remark} \label{her}
For $m=1$ we would obtain
\begin{equation} \label{eq:jor}
\lambda_{n}(\zeta)\leq \frac{n-1}{\widehat{w}_{n}(\zeta)-n}.
\end{equation}
However, this is trivially implied by (unconditioned) known facts on the exponents. 
The right hand side of \eqref{eq:jor} is larger than one by \eqref{eq:bugschl}, but on the other hand if 
$\lambda_{n}(\zeta)>1$ then $\lambda_{1}(\zeta)>2n-1$ and further $\widehat{w}_{n}(\zeta)=n$,
by a combination of~\cite[Theorem~1.6 and Theorem~1.12]{schlei}.
In other words $\widehat{w}_{n}(\zeta)>n$ implies $\lambda_{n}(\zeta)\leq 1$ anyway, 
which is indeed stronger than \eqref{eq:jor}. By the same argument
\eqref{eq:jucken} only yields something new in case of 
\[
\widehat{w}_{n}(\zeta)> \frac{n}{2}+\sqrt{\frac{n^{2}}{4}+n-1},
\]
which is of order $n+1-o(1)$. 
\end{remark}

The special case $n=2$ is of interest because then \eqref{eq:wika} is in particular satisfied whenever
$\widehat{w}_{2}(\zeta)>2$, and yields upper bounds for $\lambda_{3}(\zeta)$.

\begin{corollary} \label{orcor}
Let $\zeta$ be a real transcendental number. Assume that at least one of the two estimations
\begin{equation} \label{eq:or}
w_{1}(\zeta)<2, \qquad \text{or} \qquad \widehat{w}_{2}(\zeta)>2
\end{equation}
is satisfied. Then 
\[
\lambda_{3}(\zeta)\leq \min\left\{\frac{w_{2}(\zeta)}{\widehat{w}_{2}^{2}(\zeta)},1\right\}.
\]
In particular
\begin{equation} \label{eq:impro}
\lambda_{3}(\zeta)\leq \min\left\{\frac{1}{(\widehat{w}_{2}(\zeta)-2)\widehat{w}_{2}(\zeta)},1\right\},
\end{equation}
if we agree on $1/0=\infty$.
\end{corollary}  

\begin{proof}
The left bounds follow from \eqref{eq:hoemsch} and \eqref{eq:jucken} respectively if we can 
show that both assumptions in \eqref{eq:or} individually
imply \eqref{eq:wika} for $n=2$, that is $\widehat{w}_{2}(\zeta)>w_{1}(\zeta)$.
In case of $w_{1}(\zeta)<2$ this is obvious since $w_{1}(\zeta)<2\leq \widehat{w}_{2}(\zeta)$ by \eqref{eq:wmono}. 
In the other case $\widehat{w}_{2}(\zeta)>2$, observe that \eqref{eq:unte} with $m=1,n=2$ implies
\[
\min\{ w_{1}(\zeta), \widehat{w}_{2}(\zeta)\}\leq 2,
\]
such that indeed $\widehat{w}_{2}(\zeta)>2$ implies $w_{1}(\zeta)\leq 2<\widehat{w}_{2}(\zeta)$. 
The upper bound $1$ follows from both assumptions (individually) by a
combination of~\cite[Theorem~1.6 and Theorem~1.12]{schlei}, see the explanations subsequent 
to Remark~\ref{her}.
\end{proof}

In particular upon the conditions of the corollary we have $\lambda_{3}(\zeta)\leq \min\{ w_{2}(\zeta)/4,1\}$, 
which improves $\lambda_{3}(\zeta)\leq \lambda_{2}(\zeta)\leq \min\{(w_{2}(\zeta)-1)/2,1\}$ 
derived from \eqref{eq:khintchine}. We check that \eqref{eq:impro} improves the bound $1$ when
$\widehat{w}_{2}(\zeta)>1+\sqrt{2}\approx 2.4142$, which is still reasonably smaller than $\gamma$.
We obtain the optimal bound in \eqref{eq:impro}
when $\zeta$ is an extremal number as defined in Section~\ref{sek1},
namely $\lambda_{3}(\zeta)\leq \sigma\approx 0.6180$.
The precise value $\lambda_{3}(\zeta)=1/\sqrt{5}\approx 0.4472$ for any extremal number $\zeta$
was established in \cite{ichlondon}.
For $\zeta$ any Sturmian continued fraction, which
also satisfy $\widehat{w}_{2}(\zeta)>2$, the precise values of $\lambda_{3}(\zeta)$ 
established in~\cite{ichstw} are also strictly smaller than the bounds resulting
from Corollary~\ref{orcor}. However, there are many other real numbers that satisfy $\widehat{w}_{2}(\zeta)>2$,
as carried out in Section~\ref{hauptres}. We further remark that the estimate
\begin{equation} \label{eq:unruhe}
\lambda_{3}(\zeta)\leq \frac{w_{2}(\zeta)}{\widehat{w}_{2}^{2}(\zeta)}
\end{equation}
can be false when \eqref{eq:or} is not satisfied. Bugeaud~\cite[Corollary~1]{bug}
constructed for given $n$ and any parameter $\lambda\geq 2n-1$ numbers 
$\zeta$ which satisfy $w_{1}(\zeta)=w_{2}(\zeta)=\cdots=w_{n}(\zeta)=\lambda$. He showed 
that these numbers satisfy $\lambda_{n}(\zeta)=(w_{n}(\zeta)-n+1)/n$, i.e. there is equality
in the right inequality \eqref{eq:khintchine}. For $n=3$ any parameter $\lambda>8$ leads to a contradiction
of \eqref{eq:unruhe}. Clearly also $\lambda_{3}(\zeta)\leq 1$ is in general false, by 
~\cite[Theorem~1.6]{schlei} the condition is equivalent to $\lambda_{1}(\zeta)\leq 5$.

Now we want to discuss the case $m>n$ of Theorem~\ref{bedthm}. 

\begin{corollary} \label{tzja}
Let $\zeta$ be a real transcendental number which satisfies \eqref{eq:or}
and is not a $U_{2}$-number. Then for 
sufficiently large $m\geq m_{0}(\zeta)$ we have
\begin{equation}  \label{eq:reizmich}
\lambda_{m}(\zeta)\leq \frac{1}{\widehat{w}_{2}(\zeta)}\leq \frac{1}{2}.
\end{equation}
In particular any irrational real number $\zeta$ which satisfies
$\lambda_{1}(\zeta)<2$ and $\lim_{n\to\infty} \lambda_{n}(\zeta)>1/2$ is a $U_{2}$-number.
More generally, if $n\geq 2$ is a fixed integer and $\zeta$ is a transcendental real number which
satisfies $w_{n-1}(\zeta)<n$ and $\lim_{m\to\infty} \lambda_{m}(\zeta)>1/n$,
then $\zeta$ is a $U_{n}$-number.
\end{corollary}

\begin{proof}
We have noticed in the proof of Corollary~\ref{orcor} that the assumption \eqref{eq:or} implies \eqref{eq:wika}.
Thus the claim \eqref{eq:reizmich} follows from Theorem~\ref{bedthm} 
for sufficiently large $m$ under the assumption $w_{2}(\zeta)<\infty$ since $\widehat{w}_{m}(\zeta)\geq m$ 
by \eqref{eq:wmono}. The more general claim follows similarly.
\end{proof}

We point out that $\zeta$ that satisfy the assumptions $\lambda_{1}(\zeta)<2$ and 
$\lim_{m\to\infty} \lambda_{m}(\zeta)>1/2$ were constructed by Bugeaud, more concretely
in~\cite[Theorem~4]{bug} he constructed $\zeta$ for which $\lambda_{m}(\zeta)=1$ for all $m\geq 1$. 
In view of Corollary~\ref{tzja} all these numbers have to be $U_{2}$-numbers.

We formulate another corollary of Theorem~\ref{bedthm}, 
which is not particularly strong as such but motivates a deeper study as we will carry out below. 

\begin{corollary} \label{mcgoerg}
Let $\zeta$ be a real transcendental numbers for which  \eqref{eq:wika} is satisfied
for infinitely many $n$. Then
\begin{equation} \label{eq:trutz}
\lim_{n\to\infty} \lambda_{n}(\zeta)=0.
\end{equation}
\end{corollary}

\begin{proof}
For every $n$ which satisfies \eqref{eq:wika}, if we choose $m$ sufficiently large then Theorem~\ref{bedthm}
and \eqref{eq:wmono} imply $\lambda_{m+n-1}(\zeta)\leq \widehat{w}_{n}(\zeta)^{-1}\leq 1/n$. It remains to consider
arbitrarily large $n$.
\end{proof}

As indicated the claim is weak in the sense that upon the assumption \eqref{eq:wika}
the condition
\begin{equation} \label{eq:glei}
\lim_{n\to\infty} \frac{\widehat{w}_{n}(\zeta)}{n}=1
\end{equation}
already suffices for the deduction of \eqref{eq:trutz}, in view of \eqref{eq:khintchine}.
We strongly expect that the left hand side in \eqref{eq:glei} does not exceed $1$ for any real $\zeta$,
although we cannot even exclude that the limes inferior takes any value in $[1,2]$ by looking solely at the bounds 
in \eqref{eq:bugschl}. If we do not assume \eqref{eq:glei}, then
\eqref{eq:khintchine} is too weak to show \eqref{eq:trutz} even for certain $S$-numbers, 
more precisely it fails to decide the cases of numbers for which $\liminf_{n\to\infty} w_{n}(\zeta)/n>1$. 
Hence the assumption \eqref{eq:wika} alone still does not suffice to obtain \eqref{eq:trutz} for every $S$-number.

The main reason for the statement of Corollary~\ref{mcgoerg} is to motivate a deeper
investigation of the question how the condition \eqref{eq:wika} can be weakened or modified
to still obtain \eqref{eq:trutz}.

\begin{problem}
Find a weaker condition which implies \eqref{eq:trutz}.
Do we have \eqref{eq:trutz} for every $S$-number?
\end{problem}

If a $T$-number satisfies \eqref{eq:wika} for arbitrarily large $n$, then \eqref{eq:trutz} holds
since \eqref{eq:glei} follows from Theorem~\ref{neu} and was already established
in~\cite[Corollary~2.5]{buschl}. However, for $T$-numbers the condition \eqref{eq:wika} might be violated 
and most likely cannot be dropped for \eqref{eq:trutz} in general. 

Finally we combine Theorem~\ref{bedthm} with \eqref{eq:ichdeb}.

\begin{theorem} \label{dochin}
Keep the notation and assumptions of Theorem~\ref{bedthm}.
Then 
\begin{equation}  \label{eq:donaldtrump}
\widehat{w}_{m+n-1}^{\ast}(\zeta)\geq 
\min\left\{\frac{\widehat{w}_{m}(\zeta)\widehat{w}_{n}(\zeta)}{w_{n}(\zeta)},\widehat{w}_{n}(\zeta)\right\}
\end{equation}
and
\[
\widehat{w}_{m+n-1}^{\ast}(\zeta)\geq \min\left\{\frac{(\widehat{w}_{n}(\zeta)-n)\widehat{w}_{m}(\zeta)}{n-1},
\widehat{w}_{n}(\zeta)\right\}.
\]
In particular
\[
\widehat{w}_{2n-1}^{\ast}(\zeta)\geq \frac{\widehat{w}_{n}(\zeta)^{2}}{w_{n}(\zeta)}, \qquad 
\widehat{w}_{2n-1}^{\ast}(\zeta)\geq \frac{(\widehat{w}_{n}(\zeta)-n)\widehat{w}_{n}(\zeta)}{n-1}.
\]
\end{theorem}

We remark a certain consequence of Theorem~\ref{dochin}.
For $\zeta$ a Liouville number we have
\begin{equation} \label{eq:frit}
\widehat{w}_{m}^{\ast}(\zeta)=1, \qquad  \text{for all} \quad m\geq 1.
\end{equation}
This is an immediate consequence of~~\cite[Corollary~2.5]{buschl}, which more generally
asserted $\widehat{w}_{m}^{\ast}(\zeta)\leq n$ for any $U_{n}$-number $\zeta$ and all $m\geq 1$.
Theorem~\ref{dochin} shows a partial converse. We cannot have \eqref{eq:frit} when $w_{n-1}(\zeta)<n$ for some $n\geq 2$, 
unless possibly if $\zeta$
is a $U_{n}$-number. Indeed, if $w_{n-1}(\zeta)<n\leq \widehat{w}_{n}(\zeta)$,
the estimate \eqref{eq:donaldtrump} and \eqref{eq:wmono} yield 
\[
\widehat{w}_{m+n-1}^{\ast}(\zeta)\geq 
\min\left\{\frac{\widehat{w}_{m}(\zeta)\widehat{w}_{n}(\zeta)}{w_{n}(\zeta)},\widehat{w}_{n}(\zeta)\right\}\geq 
\min\left\{\frac{n}{w_{n}(\zeta)}\cdot m, n\right\}.
\]
The right hand side obviously equals $n\geq 2$ for large $m$ unless $w_{n}(\zeta)=\infty$,
which clearly prohibits \eqref{eq:frit}. For $U_{n}$-numbers, on the other hand, as remarked above we have
$\widehat{w}_{m}^{\ast}(\zeta)\leq n$ for all $m\geq 1$.

\section{Proofs} \label{proofs}

We will utilize the relation
\begin{equation} \label{eq:hoehe}
\frac{1}{K(n)}H(P)H(Q)\leq H(PQ)\leq K(n)H(P)H(Q)
\end{equation}
valid for all $P,Q\in\mathbb{Z}[T]$ of degree at most $n$ and a constant $K(n)$, as noticed 
by Wirsing~\cite{wirsing}. Hence if positive $\alpha,\beta$ are defined by 
\[
\vert P(\zeta)\vert =H(P)^{-\alpha}, \qquad \vert Q(\zeta)\vert =H(Q)^{-\beta}
\]
then 
\[
\vert P(\zeta)Q(\zeta)\vert \geq (H(P)H(Q))^{-\max\{\alpha,\beta\}} \gg_{n} H(PQ)^{-\max\{\alpha,\beta\}}.
\]
In particular it follows that we may consider the polynomials from the definition
of $w_{n}(\zeta)$ irreducible (but not necessarily of exact degree $n$), see~\cite[Hilfssatz~4]{wirsing}.
Moreover, if either $n=1$ or $w_{n}(\zeta)>w_{n-1}(\zeta)$, 
then the corresponding polynomials have to have degree precisely $n$.
This fact was already frequently used in \cite{buschl},
and will of significant importance in the subsequent proof of Theorem~\ref{neu}.

\begin{proof}[Proof of Theorem~\ref{neu}]
Let $l\in\{1,2,\ldots,m\}$ be the smallest integer such that $w_{l}(\zeta)=w_{m}(\zeta)$.
We will show
\begin{equation} \label{eq:nonedda}
w_{l+n-1,l+n}(\zeta)\geq \min\{ w_{l}(\zeta),\widehat{w}_{n}(\zeta)\}=\min\{ w_{m}(\zeta),\widehat{w}_{n}(\zeta)\}.
\end{equation}
By \eqref{eq:alignzwei} applied in dimension $l+n-1$ and with $j=l+n$, 
and the trivial estimates \eqref{eq:ldiri}, we indeed deduce
\[
\widehat{\lambda}_{m+n-1}(\zeta)\leq \widehat{\lambda}_{l+n-1}(\zeta)
\leq \frac{1}{\min\{w_{l}(\zeta),\widehat{w}_{n}(\zeta)\}}=
\max \left\{ \frac{1}{w_{m}(\zeta)},\frac{1}{\widehat{w}_{n}(\zeta)}\right\}.
\]

Let $\epsilon\in(0,1)$. As carried out above and by our definition of $l$, there exist
{\em irreducible} polynomials $P$ of degree precisely $l$ that satisfy
\begin{equation} \label{eq:numerouno}
H(P)=X, \qquad \vert P(\zeta)\vert \leq X^{-w_{l}(\zeta)+\epsilon}
\end{equation}
as in the definition of $w_{l}(\zeta)$. 
For $K(n)$ the constant in \eqref{eq:hoehe}, let
\[
Y=\frac{X}{2K(n)}.
\]
Then by definition of $\widehat{w}_{n}$ 
there must exist a solution $Q\in\mathbb{Z}[T]$ of degree at most $n$ to
\[
H(Q)\leq Y, \quad 
\vert Q(\zeta)\vert \leq Y^{-\widehat{w}_{n}(\zeta)+\epsilon}. 
\]
We may assume the degree $d$ of $Q$ is precisely $d=n$, otherwise replace $Q(T)$ by $T^{n-d}Q(T)$.
Together with the estimate $\widehat{w}_{n}(\zeta)\leq 2n-1$ implied by \eqref{eq:bugschl} and
$\epsilon<1$, we have
\begin{equation} \label{eq:numerodos}
H(Q)\leq Y, \qquad 
\vert Q(\zeta)\vert 
\leq (2K(n))^{2n}\cdot X^{-\widehat{w}_{n}(\zeta)+\epsilon}. 
\end{equation}
By construction and \eqref{eq:hoehe} this polynomial $Q$ cannot be a multiple of $P$. 
Hence, since $P$ is irreducible, the polynomials $P$ and $Q$ have no common factor over $\mathbb{Z}[T]$.
Moreover, any such pair
$(P,Q)$ satisfies the properties
\[
\max \{ H(P),H(Q)\} =X, \qquad 
\max \{ \vert P(\zeta)\vert,\vert Q(\zeta)\vert \}\leq (2K(n))^{2n}\cdot X^{-\widehat{w}_{n}(\zeta)+\epsilon}.
\]
Consider the set $\mathscr{P}$ of $l+n$ polynomials
\begin{equation} \label{eq:p}
\mathscr{P}=\{ P,TP,T^{2}P,\ldots,T^{n-1}P, Q, TQ,\ldots,T^{l-1}Q\}. 
\end{equation}
The determinant of the $(l+n)\times (l+n)$ matrix whose rows are given by the coefficients of the polynomials
in $\mathscr{P}$ is just the resultant of $P$ and $Q$, which is non-zero since we have noticed
that $P$ and $Q$ have no common factor. Hence $\mathscr{P}$ is linearly independent and spans
the space of polynomials of degree at most $l+n-1$.
Since any $R\in\mathscr{P}$ arises either from $P$ or $Q$ by multiplication of $T^{k}$ with $k\leq n-1$
from \eqref{eq:numerouno} and \eqref{eq:numerodos} we infer
\begin{equation} \label{eq:oppen}
H(R)\leq X, \qquad \vert R(\zeta)\vert \leq C(n,\zeta)\cdot X^{-\widehat{w}_{n}(\zeta)+\epsilon}
\end{equation}
with $C(n,\zeta)=\max\{1,\vert\zeta\vert^{2n-1}\}(2K(n))^{2n}$ 
for every $R\in\mathscr{P}$. Thus, and since $\mathscr{P}$ is linearly independent, 
we have established \eqref{eq:nonedda} as $\epsilon$
can be chosen arbitrarily small.
\end{proof}

\begin{remark} \label{schauher}
As indicated in Section~\ref{hauptres}, for $m=n$ the polynomials implicitly given 
by the definition of $\widehat{w}_{n}$ are closely related to the polynomials $P_{i}$ involved
in the proof by Laurent~\cite{laurent}. He relates these polynomials $P_{i}$ to the best approximation vectors
associated to $\zeta$ via Hankel matrices. 
\end{remark}

We carry out how to modify the proof to obtain Theorem~\ref{conditio}.

\begin{proof}[Proof of Theorem~\ref{conditio}]
By definition of $w_{n,2}$ there exist arbitrarily large $X$ such that 
\[
\max\{H(P),H(Q)\}=X, \qquad \max\{ \vert P(\zeta)\vert, \vert Q(\zeta)\vert\}\leq X^{-w_{n,2}(\zeta)+\epsilon}.
\]
It follows from \eqref{eq:hoehe} and the assumption \eqref{eq:bedin} that any such pair both
$P$ and $Q$ of sufficiently large heights $H(P),H(Q)$ must be irreducible of degree precisely $n$. 
Now we consider again the sets $\mathscr{P}$ as in \eqref{eq:p} with $l=n$ arising from these pairs to verify
\begin{equation} \label{eq:noneddard}
w_{2n-1,2n}(\zeta)\geq w_{n,2}(\zeta).
\end{equation}
This again implies the claim by \eqref{eq:alignzwei}.
\end{proof}

Now we prove Theorem~\ref{lauverb}. The proof employs the reformulation
\begin{equation} \label{eq:wepper}
\widehat{\lambda}_{n}(\zeta)\leq -\frac{n-2+(n-1)\lambda_{n}(\zeta)}{2}+
\sqrt{ \left(\frac{n-2+(n-1)\lambda_{n}(\zeta)}{2}\right)^{2}+(n-1)\lambda_{n}(\zeta)}
\end{equation}
of \eqref{eq:umform}, in dimension $2n$.

\begin{proof}[Proof of Theorem~\ref{lauverb}]
Application of Theorem~\ref{neu} with suitable $m,n$ yields
\[
\widehat{\lambda}_{2n}(\zeta)\leq \max\left\{\frac{1}{w_{n}(\zeta)},\frac{1}{\widehat{w}_{n+1}(\zeta)}\right\}.
\]
If we assume that the right hand side of the maximum is larger, then from \eqref{eq:wmono}
we infer the bound $1/(n+1)$ which can be readily checked to be smaller than $\theta_{n}$.
Hence we may assume
\begin{equation} \label{eq:landa}
\widehat{\lambda}_{2n}(\zeta)\leq \frac{1}{w_{n}(\zeta)}.
\end{equation}
On the other hand, combination of \eqref{eq:khintchine} and \eqref{eq:ldirich} gives
\begin{equation} \label{eq:zorn}
\lambda_{2n}(\zeta)\leq \lambda_{n}(\zeta)\leq \frac{w_{n}(\zeta)-n+1}{n}.
\end{equation}
For $k$ an integer define the function
\[
\Phi_{k}(x)=-\frac{k-2+(k-1)\frac{x-n+1}{n}}{2}+
\sqrt{ \left(\frac{k-2+(k-1)\frac{x-n+1}{n}}{2}\right)^{2}+(k-1)\frac{x-n+1}{n}}.
\]
We insert the right hand side of \eqref{eq:zorn}
in \eqref{eq:wepper} in dimension $2n$, and by definition of $\Phi_{2n}$ we obtain
\begin{equation} \label{eq:lammda}
\widehat{\lambda}_{2n}(\zeta)\leq \Phi_{2n}(w_{n}(\zeta)).
\end{equation}
Since the bound in \eqref{eq:landa} decays whereas 
$\Phi_{2n}$ increases as $w_{n}(\zeta)$ increases in the interval $[n,\infty]$,
and $\Phi_{2n}(n)<1/n$ follows easily from insertion in \eqref{eq:wepper}, 
the maximum upper bound is attained when the right hand sides of \eqref{eq:landa} and \eqref{eq:lammda}
are equal. Some rearrangements show that this is the case if $w_{n}(\zeta)$ is the root of
the cubic polynomial
\[
Q_{n}(T)= \frac{2n-1}{n}T^{3}+(-2n+1)T^{2}+\frac{1-n}{n}T-1
\]
in the interval $(n,\infty)$. Denote by $\overline{Q}_{n}$ the polynomial that 
arises from $Q_{n}$ above by reversing the order of the coefficients, which equals the polynomial $P_{n}$
in \eqref{eq:polynom} up to the factor $-1$. It is well known that $T^{3}Q_{n}(1/T)=\overline{Q}_{n}(T)$
and thus if $\{ \beta_{1},\beta_{2},\beta_{3}\}$ are the roots $Q_{n}$ in $\mathbb{C}$ then
the reciprocals $\{ \beta_{1}^{-1},\beta_{2}^{-1},\beta_{3}^{-1}\}$ are the roots of $\overline{Q}_{n}$ and $P_{n}$.
By \eqref{eq:landa}, one of these roots $\beta_{i}$ and $\beta_{i}^{-1}$ respectively, coincides
with $\theta_{n}^{-1}$ and $\theta_{n}$, respectively. The argument also shows that the bound 
$\theta_{n}$ is strictly smaller than $1/n$.
\end{proof}

We turn towards the bounds for the exponents $\lambda_{n}$.
The proof of Theorem~\ref{bedthm} essentially follows from the following lemma.

\begin{lemma} \label{lemur}
Keep the notation and assumptions of Theorem~\ref{bedthm}. 
Let
\[
v:=\frac{\widehat{w}_{m}(\zeta) \widehat{w}_{n}(\zeta)}{w_{n}(\zeta)}
\]
and 
\[
w:= 
\min\left\{v,\widehat{w}_{n}(\zeta)\right\}.
\]
Then, for any  $\varepsilon>0$ and every large $X\geq X_{0}(\varepsilon)$ we can find coprime polynomials $P$ and $Q$
of degrees at most $n$ and $m$ respectively, such that 
\begin{equation} \label{eq:baldfertig}
\max\{ H(P),H(Q)\}\leq X, \qquad 
\max\{ \vert P(\zeta)\vert, \vert Q(\zeta)\vert\}\leq X^{-w+\varepsilon}.
\end{equation}
\end{lemma}

\begin{proof}
Let $\epsilon>0$ and $X\geq X_{0}(\epsilon)$ be a large parameter. 
Then by definition of $\widehat{w}_{n}$ there exists a polynomial $P\in\mathbb{Z}[T]$ 
of degree at most $n$ such that
\begin{equation} \label{eq:einteil}
H(P)\leq X, \qquad \vert P(\zeta)\vert \leq X^{-\widehat{w}_{n}(\zeta)+\epsilon}. 
\end{equation}
By assumption \eqref{eq:wika} we know every such $P$ must be irreducible of precise degree $n$.
On the other hand by definition of $w_{n}$ we have
\begin{equation} \label{eq:herxn}
\vert P(\zeta)\vert \geq H(P)^{-w_{n}(\zeta)-\epsilon}\geq X^{-w_{n}(\zeta)-\epsilon}. 
\end{equation}
Choose $\mu=(\widehat{w}_{n}(\zeta)-2\epsilon)/(w_{n}(\zeta)+\epsilon)$ and let $Y=X^{\mu}$.
Obviously $\mu<1$ and $Y<X$.
Then \eqref{eq:herxn} becomes
\begin{equation} \label{eq:piscis}
\vert P(\zeta)\vert \geq Y^{-\widehat{w}_{n}(\zeta)+2\epsilon}. 
\end{equation}
Now by definition of $\widehat{w}_{m}$ there must be a polynomial $Q\in\mathbb{Z}[T]$ 
of degree at most $m$ such that
\begin{equation} \label{eq:zweiteil}
H(Q)\leq Y, \qquad \vert Q(\zeta)\vert \leq Y^{-\widehat{w}_{m}(\zeta)+\epsilon}.
\end{equation}
We claim that $P,Q$ must be coprime. Since $P$ is irreducible
we just have to exclude that $Q$ is a polynomial multiple of $P$. Assume that otherwise $PR=Q$
for some $R\in\mathbb{Z}[T]$.
From combination of \eqref{eq:einteil} and \eqref{eq:herxn} it follows that
\[
H(P)\geq X^{\frac{\widehat{w}_{n}(\zeta)-\epsilon}{w_{n}(\zeta)+\epsilon}}=X^{\mu+\theta}
\]
where $\theta=\epsilon/(w_{n}(\zeta)+\epsilon)>0$. On the other hand $H(Q)\leq Y=X^{\mu}$. The assumption $PR=Q$ leads
to a contradiction in view of \eqref{eq:hoehe} if $X$ was chosen large enough. Thus the claim is shown.
Moreover combination of \eqref{eq:einteil} and \eqref{eq:zweiteil} yields 
\[
\max\{H(P),H(Q)\} \leq X, \qquad \max\{ \vert P(\zeta)\vert,\vert Q(\zeta)\vert\}\leq 
\max\{ X^{-\widehat{w}_{n}(\zeta)+\epsilon}, X^{\mu(-\widehat{w}_{m}(\zeta)+\epsilon)}\}.
\]
For the latter exponent in the right hand side we have
\[
\mu(\widehat{w}_{m}(\zeta)-\epsilon)\geq v-\epsilon_{1}
\]
for $\epsilon_{1}>0$ some variation of $\epsilon$ which tends to $0$ as $\epsilon$ does.
Thus for $\varepsilon=\max\{\epsilon,\epsilon_{1}\}$ the polynomials $P,Q$ have all desired properties
including \eqref{eq:baldfertig}.
\end{proof}

\begin{remark}
In particular for $m=n$ we obtain
\[
\widehat{w}_{n,2}(\zeta)\geq \frac{\widehat{w}_{n}(\zeta)^{2}}{w_{n}(\zeta)}.
\]
This refines $w_{n,3}(\zeta)\geq \widehat{w}_{n}(\zeta)^{2}/w_{n}(\zeta)$ from~\cite[Theorem~4.2]{ichglasgow}
upon the assumption of \eqref{eq:wika}.
Concerning the natural generalization
$\widehat{w}_{n,j}(\zeta)\geq \widehat{w}_{n}(\zeta)^{j}/w_{n}(\zeta)^{j-1}$ of \cite[Theorem~4.2]{ichglasgow}
for $j\geq 3$, to guarantee the linear independence of the involved polynomials becomes a serious problem.
\end{remark}

Now we can prove Theorem~\ref{bedthm}.

\begin{proof}[Proof of Theorem~\ref{bedthm}]

Let $w$ be as in Lemma~\ref{lemur}.
We have shown in Lemma~\ref{lemur} that for every $\epsilon>0$ and every large parameter $X$ there
exist two coprime polynomials $P,Q$ of degree at most $n$ and $m$ respectively such that
\eqref{eq:baldfertig} is satisfied.
Now we can again consider the set $\mathscr{P}$ as in \eqref{eq:p} arising from such $P,Q$,
which provides $m+n$ linearly independent polynomials of degree
at most $m+n-1$, and moreover \eqref{eq:baldfertig} implies
\[
H(R)\leq X, \qquad 
\vert R(\zeta)\vert\ll_{m,n,\zeta} X^{-w+\epsilon},
\]
for all $R\in\mathscr{P}$. Thus 
\[
\widehat{w}_{m+n-1,m+n}(\zeta)\geq w.
\]
Together with \eqref{eq:aligneins} in dimension $m+n-1$ and $j=1$, we infer
\[
\lambda_{m+n-1}(\zeta)\leq w^{-1},
\]
which is the claim.
\end{proof}


\begin{thebibliography}{99} 


\bibitem{bertis} V. I. Bernik and K.I. Tishchencko. Integral polynomials with an 
overfall of the coefficient values and Wirsing's problem. 
{\em Dokl. Akad. Nauk Belarusi} 37 (1993), no. 5, 9--11 (in Russian).

\bibitem{bugbuch} Y. Bugeaud. Approximation by algebraic numbers.
{\em Cambridge Tracts in Mathematics}, Cambridge 2004. 

\bibitem{bug} Y. Bugeaud. On simultaneous rational approximation to a real number
and its integral powers. {\em Ann. Inst. Fourier (Grenoble)} 60 (2010), 2165--2182.



\bibitem{buglau} Y. Bugeaud and M. Laurent. Exponents of Diophantine approximation
and Sturmian continued fractions. {\em Ann. Inst. Fourier (Grenoble)} 55 (2005), no. 3, 773--804.

\bibitem{buschl} Y. Bugeaud and J. Schleischitz. On uniform approximation to real numbers.
{\em Acta Arith.} 175 (2016), no. 3, 255--268. 

\bibitem{cassaigne} J. Cassaigne. Limit values of the recurrence sequence of Sturmian sequences.
{\em Theor. Comput. Sci.} 218 (1999), 3--12.


\bibitem{davsh} H. Davenport and W. M. Schmidt. 
Approximation to real numbers by algebraic
integers. {\em Acta Arith.} 15 (1969), 393--416.

\bibitem{fisch1} S. Fischler. Spectres pour l'approximation d'un nombre reel et
de son carree. {\em C.R. Math. Acad. Sci. Paris} 339 , no. 10 (2004), 679--682.

\bibitem{fisch2} S. Fischler. Palindromic prefixes and Diophantine approximation.
{Monatshefte Math.} 151, no. 1 (2007), 1--87.

\bibitem{german} O. German. On Diophantine exponents and Khintchine's transference principle. 
{\em Mosc. J. Comb. Number Theory} 2 (2012), 22--51.

\bibitem{khintchine} A.Y. Khintchine, \"Uber eine Klasse linearer diophantischer Approximationen,
{\em Rend. Circ. Mat. Palermo} 50 (1926), 706--714.

\bibitem{laurent} M. Laurent. Simultaneous rational approximation to the successive
powers of a real number. {\em Indag. Math.} 11 (2003), 45---53. 

\bibitem{leveque} 
W. J. Leveque. On Mahler's U-numbers, {\em J. London Math. Soc.} 28 (1953), 220--229.

\bibitem{damroy} D. Roy. On simultaneous rational approximations to a real number, its square, and its cube.
{\em Acta Arith.} 133 (2008), 185--197.


\bibitem{royyy} D. Roy, Approximation to real numbers by cubic algebraic integers I, 
{\em Proc. London Math. Soc.} 88 (2004), 42--62.
  

\bibitem{royexp} D. Roy. On two exponents of approximation related to a real number and its square.
{\em Canad. J. Math.} 59 (2007), 211--224. 

\bibitem{ichlondon} J. Schleischitz. Approximation to an extremal number, its square and its cube.
{\em to appear in Pacific J. Math., arXiv: 1602.04731}

\bibitem{schlei} J. Schleischitz. On the spectrum of Diophantine approximation constants.
{\em Mathematika} 62 (2016), 79--100.

\bibitem{j2} J. Schleischitz. Two estimates concerning classical Diophantine approximation constants,
{\em Publ. Math. Debrecen} 84/3-4 (2014), 415--437. 

\bibitem{ichstw} J. Schleischitz. Determination of some exponents of approximation for Sturmian 
continued fractions. {\em arXiv:} 1603.08808.

\bibitem{ichglasgow} J. Schleichitz. Some notes on the regular graph defined by Schmidt and Summerer
and uniform approximation. {\em to appear in JP J. Algebra Number Theory Appl., arXiv:} 1601.00842.

\bibitem{tnumbers} W.M. Schmidt. T-numbers do exist. 
{\em Symposia Math. IV}, Inst. Naz. di Alta Math. Academic Press,
Rome (1968), 3--26. 

\bibitem{ss} W.M. Schmidt, L. Summerer. Parametric geometry of numbers and applications, 
{\em Acta Arith.} 140 (2009), no. 1,  67--91.  


\bibitem{ssch} W.M. Schmidt, L. Summerer. 
Diophantine approximation and parametric geometry of numbers. 
{\em Monatsh. Math.} 169 (2013), 51--104. 

\bibitem{sums}  W.M. Schmidt, L. Summerer. 
Simultaneous approximation to three numbers.
{\em Mosc. J. Comb. Number Theory} 3 (2013), 84--107.

\bibitem{tish} K.I. Tishchenko,
On approximation of real numbers by algebraic numbers of bounded degree. 
{\em J. Number Theory} 123 (2007), 290--314.

\bibitem{wirsing} E. Wirsing. Approximation mit algebraischen Zahlen beschr\"ankten Grades.
{\em J. Reine Angew. Math.}  {206} (1961), 67--77.

\end{thebibliography}
\end{document}